\documentclass[11pt]{amsart}

\usepackage{a4}
\usepackage[all,cmtip]{xy}
\usepackage{amsmath}
\usepackage{amsthm}
\usepackage{amssymb}
\usepackage{amscd,latexsym}
\usepackage{graphicx}

\newtheorem{theorem}{Theorem}[section]

\newtheorem{lemma}[theorem]{Lemma}
\newtheorem{proposition}[theorem]{Proposition}
\newtheorem{corollary}[theorem]{Corollary}

\newtheorem{remark}[theorem]{Remark}

\newtheorem*{mainthm}{Main Theorem}

\theoremstyle{definition}
\newtheorem{definition}[theorem]{Definition}

\newtheorem{definitions and remarks}[theorem]{Definitions and Remarks}

\theoremstyle{remark}

\newcommand{\IP}{{\mathbb P}}
\newcommand{\IQ}{{\mathbb Q}}
\newcommand{\IA}{{\mathbb A}}

\newcommand{\IC}{{\mathbb C}}

\newcommand{\IZ}{{\mathbb Z}}

\newcommand{\nind}{{\noindent}}

\begin{document}

\title{Invariants of Surfaces of Degree $d$ in $\IP^n$}

\author{R.V. Gurjar}
\address{Department of Mathematics\\
  Indian Institute of Technology Bombay\\
  Mumbai 400076\\
  India}
\email{gurjar@math.iitb.ac.in}

\author{Alok Maharana}
\address{Department of Mathematical Sciences\\
  Indian Institute of Science Education and Research Mohali\\
  Sector 81, Mohali 140306\\
  India}
\email{maharana@iisermohali.ac.in}



\maketitle

\begin{abstract}  
{  We prove that several invariants of a possibly singular complex affine or projective variety of degree $d$ in the affine space $\IA^{n}$, or $\IP^n$, are bounded by a function of $d$ alone, provided $b_{1}=0$ for a resolution of singularities of the variety.
}
 \end{abstract}



\section{Introduction}

\nind
In this paper we will only consider complex algebraic or analytic varieties.\\

\nind 
The main theme of this paper is that for normal affine, or projective, surfaces $S$ (resp. $X$) of degree $d$ in $\IA^n$ (resp. in $\IP^n$) and $b_1(S')=0$ (resp. $b_1(X')=0$) for a resolution of singularities $S'$ of $S$ (resp. $X'$ of $X$) many invariants of $S,S',X,X'$ are bounded by functions of $d$ alone, where for one closed embedding of an affine surface $S$ in $\IA^n$ we define degree of $S$ as the degree of its closure $X$ in $\IP^n$. This result is easy to prove if $X$ is a smooth complete intersection in $\IP^n$. But for normal $X$ this is not obvious. Some of our proofs give quantitative bounds for these invariants (which are not best possible) but some bounds are given only existential proofs. Even for normal surfaces $X$ of degree $d$ in $\IP^3$ and $b_1(X')=0$ we have not been able to give quantitative bounds for order of $H_1(X';\IZ)$ since our proof uses some easy stratification arguments. Hopefully, a better proof can correct this defect.\\

We now state the main result of this paper.\\
  
\begin{mainthm} Let $S$ be an irreducible normal affine closed subvariety of degree $d$ in $\IA^{n}$ for $n>2$. Assume that a resolution of singularities $S'$ of $S$ has $b_1(S')=0$. Then the order of $H_{1}(S'; \IZ)$ is bounded by a function of $d$ alone.\\
Similarly, the order of $H_1(S-Sing~S;\IZ)$ is bounded by a function of $d$ alone.\\
In particular, the order of $H_{1}(S; \IZ)$ has the same bound.\\ 

\end{mainthm}

\nind
Throughout the paper we will simply say that some invariant is bounded by $d$, instead of bounded above by a function of $d$ alone.\\
  
  \begin{remark} Let $S$ (or $X$) be a closed normal affine (resp. projective) subvariety of $\IA^n$ (resp. $\IP^n$).  If $dim~S>2$ (resp. $dim~X>2$) then a general hyperplane section of $S$ (resp. $X$) is normal by a result of A. Seidenberg. Further, 
a Lefschetz theorem due to Hamm-Le Dung Trang (\cite{HT}) and a result in \cite{GS} imply that $H_{1}$ remains the same (Lemma 2.5). Hence we reduce to the case of $dim~S=2$ (resp. $dim~X=2$). In view of these observations we will only deal with surfaces in the rest of the paper.
  \end{remark}

\nind 
Let $S$ be as above and dim $S=2$. We will also prove the following results.\\

\nind
(1) For any singular point $p$ of $S$ the size and the entries of the intersection matrix of the exceptional curves in the minimal resolution of singularity $S'$ at $p$ are bounded by $d$.\\
Similarly, dim $H^1(S',{\mathcal O_{S'}})$ is bounded by $d$, where ${\mathcal O_{S'}}$ is the structure sheaf of $S'$.  

\nind
(2) Let $S,S'$ be as above. Then for a suitable open embedding $S'\subset X'$ into a smooth projective surface $X'$ the size and the entries of intersection matrix of the irreducible components $D_i'$ of $D':=X'-S'$ are bounded by $d$.\\   
  
\nind
(3) With $S,X,X'$ as above, order of $H_1(X';\IZ)$, $b_2(X'),K_{X'}^2,p_g(X')$ are bounded by $d$. Similarly, the rank and order of the torsion subgroup of the first homology group at infinity $H_1^{\infty}(S)$ (defined later) are bounded by $d$.\\

\nind
In the last section we have given some examples to illustrate how the hypothesis in the main theorem is necessary, and how the order of $H_1(S';\IZ)$ can tend to infinity with $d$.\\ 
  
  \section{Notations and preliminaries}
\nind
We will use the following notation.\\

  \begin{enumerate}
  \item $S$ is an irreducible closed normal affine surface in $\IA^{n}$ of degree $d$. For one closed embedding of a surface $S$ in $\IA^n$ we define degree of $S$ as the degree of its closure in $\IP^n$.
  \item $X$ is the closure of $S$ in $\IP^{n}$. 
  \item $D: = X-S$ is the boundary divisor of $S$. 
  \item $D=\sum_{i=1}^{i=r} D_{i}$ is an irreducible decomposition of $D$ where $1\leq r\leq d$. \item $P=$ the set of singular points of $X$. ($X$ may be non-normal.) 
  \item $X'$ is a resolution of singularities of $X$ such that the total transform $D'$ of $D$ is a simple normal crossing divisor in $X'$. $S'$ is a resolution of singularities of $S$. 
  \item A simple normal crossing divisor on a smooth surface will be called an {\it snc} curve.
  \item We will denote by $b_i,e(.)$ the $i^{th}$ Betti number and Euler-Poincar\'e characteristic resp. of a topological space. For a topological space $T$ and a closed subspace $Z$ of $T$ we denote the $i^{th}$ relative cohomology group with compact support for the pair by $H^i_c(T,Z)$. We will only use integer or rational coefficients for cohomology. For details about this cohomology, duality,...we refer the reader to \cite{Spanier}.
  \end{enumerate}
  
  The next result is well-known.\\
  
  \begin{lemma} Let $f: Y\to Z$ be a birational morphism of irreducible normal algebraic varieties. Then the induced homomorphism $\pi_1(Y)\to \pi_1(Z)$ is a surjection.\\
  In particular, $H_1(Y;\IZ)\to H_1(Z;\IZ)$ is a surjection.
  
  \end{lemma}
  
This applies to a Zariski-open embedding $Y\subset Z$, or to a resolution of singularities $Z'\to Z$.\\

\begin{lemma} Let $Z$ be a smooth projective irreducible surface and $D_1,D_2,..,D_r$ finitely many distinct irreducible curves on $Z$. If the homology classes $[D_i]$ are independent in $H_2(Z;\IZ)$ then $b_1(Z)=b_1(Z-\cup D_i)$.

\end{lemma}

\begin{proof} Let $S:=Z-D$, where $D=\cup D_i$. Then $H^2(D;\IZ)\cong \IZ^r$, generated freely by the cohomology classes of $D_i$. We consider the long exact cohomology sequence with integer coefficients for the pair $(Z,D)$.
$$...H^2(Z)\to H^2(D)\to H^3(Z,D)\to H^3(Z)\to (0)....$$
We have used $H^3(D)=0$ since $D$ is a curve. By assumption, $H_2(D)\to H_2(Z)$ is injective. So $H^2(Z)\to H^2(D)$ has a finite cokernel. By duality, $H^3(Z,D)\cong H_1(S)$ and $H^3(Z)\cong H_1(Z)$. Hence the result follows.

\end{proof}

  \begin{lemma}\label{b3birinv} For irreducible normal projective surfaces $Y$, the third Betti number $b_{3}(Y)$ is a birational invariant. \cite{Montreal}. 
  \end{lemma} 

\nind
This just uses the negative definiteness of the intersection form of the exceptional divisor of a normal surface singularity.\\
  
  \begin{proof}
  Let $Y$ be a normal projective surface (with smooth locus $Y^0$) and $(Y',D')$ its minimal snc resolution of singularities so that $D'$ is a disjoint union of connected components each corresponding to an snc resolution of corresponding singular point of $Y$. Consider the long exact cohomology sequence for the pair $(Y',D')$ with integral coefficients: $$H^2(Y')\xrightarrow{\phi} H^2(D')\to H^3(Y',D')\to H^3(Y')\to H^3(D')=0$$ 
  Then $H^2(Y')$ maps to $H^2(D')$ with finite cokernel because of negative definiteness of intersection matrix of each connected component of $D'$. We have  $H^3(Y',D')\cong H_1(Y^0)$. Tensoring the sequence with $\IQ$ we get $b_1(Y^0)=b_3(Y')$.
  
  Now consider the long exact cohomology sequence for the pair $(Y,\textrm{Sing}(Y))$ with integral coefficients, where $\textrm{Sing}(Y)$ is the finite set of singular points of $Y$:
  $$H^2(\textrm{Sing}(Y))\to H^3(Y,\textrm{Sing}(Y))\to H^3(Y)\to H^3(\textrm{Sing}(Y))$$
  But $\textrm{Sing}(Y)$ is a finite set so it has zero $H^2,H^3$ and since $H^3(Y,\textrm{Sing}(Y))\cong H_1(Y^0)$ we get $H_1(Y^0)\cong H^3(Y)$ which gives $b_1(Y^0)=b_3(Y)$. Since we already have $b_1(Y^0)=b_3(Y')$ the result is proved. 
  
  \end{proof}
  
 \begin{lemma} Let $D$ be a reduced curve of degree $d$ in $\IP^2$. Then $b_1(D)\leq (d-1)(d-2)$.
 
 \end{lemma} 
 
 \begin{proof} If $D$ is smooth irreducible then this follows from genus formula for $D$. For the general case let $D_1$ be a smooth irreducible curve of degree $d$ in $\IP^2$ which meets $D$ in $d^2$ distinct points. Let $Y$ be the blow-up of $\IP^2$ at these $d^2$ points. Then there is a morphism $f:Y\to\IP^1$ such that the proper transforms $D',D_1'$ of $D,D_1$ resp. are isomorphic to $D,D_1$ resp. and are full fibers of $f$. Also, $f$ has a cross-section. Since $b_1(D_1)=(d-1)(d-2)$, using the fact that $D'$ is a strong deformation retract of a saturated tubular neighborhood we infer that $b_1(D')\leq b_1(D_1')$. This proves the result.
 
 \end{proof}
 
 \nind
 We need the next result from \cite{GS}.\\ 
  
  \begin{lemma}\label{GS}
  Let Y be a normal algebraic variety and $Y'\to{Y}$ a resolution of singularities which is an isomorphism on $Y_{reg}$. Then the map $H_1(Y_{reg};\IZ)\to H_1(Y';\IZ)$ is surjective with finite kernel. 
  \end{lemma}
  
\nind
In order to reduce the proof of the main theorem to surfaces we need the following result which uses the above lemma.\\

\begin{lemma} Let $S$ be an irreducible normal affine variety of dimension $>2$. Suppose
that for a resolution of singularities $S'\to S$ we have $b_1(S')=0$. Let $S_0$ be the irreducible normal affine surface obtained by intersecting $S$ repeatedly by general hyperplane sections. Then for a resolution of singularities $S_0'\to S_0$ we have $b_1(S_0')=0$. 

\end{lemma}

\begin{proof} For simplicity we give a proof when dim $S=3$.\\
\nind
For any subset $Z$ of $S$ we denote by $Z'$ its inverse image in $S'$. By hypothesis, $H_1(S')$ is finite. Hence by Lemma ~\ref{GS}, $H_1(S_{reg})$ is finite. By \cite{HT}, $H_1(S_{reg}\cap H)\cong H_1((S\cap{H})_{reg})$. But $(S\cap H)_{reg}$ is embedded in $(S\cap H)'$ as a Zariski open set. Hence $H_1((S\cap H)')$ is finite showing that $b_1(S_0')=0$ where $S_0=S\cap H$. 
  
\end{proof}  

  \begin{lemma}\cite{Varchenko}\label{Varchenko}
  Consider a set of $q$ homogeneous polynomials
  $$T_i^1(z_1^0,\dots,z_{n+1}^0,t),\dots T_i^{m_i}(z_1^0,\dots,z_{n+1}^0,t);\  i=1,\dots,q$$
  in the variables $z_1^0,\dots,z_{n+1}^0$ whose coefficients depend polynomially on $t\in\IC^l$. Consider the direct product $\IP^n\times\IC^l$ with projection $p$ on the second factor. In the space $\IP^n\times\IC_j^l$ distinguish $q$ subsets $E^1,\dots,E^q$, 
  $$E^i=\{(z_1^0:\dots:z_{n+1}^0,t): T_i^j(z_1^0,\dots,z_{n+1}^0,t)=0\}$$
  Then there exists a proper algebraic subset $A$ of $\IC^l$ such that the family $(p,\IP^n\times\IC^l,E^1,\dots,E^l)$ is equisingular outside $A$, i.e. the family $(p,\IP^n\times\IC^l,E^1,\dots,E^l)$ restricted to the inverse image of $\IC^l\setminus{A}$ is a (topologically) locally trivial family. 
  
  \end{lemma}
  
  \section{Smooth completion case}
  
  We first prove the main theorem when $X$ is a smooth surface in $\IP^3$. 
  
  \noindent As in the introduction, $X$ is the completion of an affine surface $S\subset\IA^3$.  
  
  \begin{theorem} If the completion $X$ of $S$ is smooth then $\pi_{1}(S)=0$, $H^{3}(X)=0$ and $b_{2}(S)\leq (d-1)^{3}$. 
  \end{theorem}
  
  \begin{proof} By a result of Nori (Example 6.8 in his paper on Zariski conjecture \cite{Nori}) we know that $\pi_{1}(S)=0$. Since $S$ is an affine surface, the third and fourth integral homology groups vanish for $S$ owing to the fact that the homotopy type of $S$ is a CW-complex of real dimension at most $2$. There is no torsion in $H^{2}(S)\cong Hom(H_{2}(S);\IZ)$ since $H_{1}(S)=0$. We give a bound on $b_{2}(S)$ to finish. 
  For any smooth hypersurface $X$ of degree $d$, $b_{2}(X)=d^{3}-4d^{2}+6d-2$. 
  The long exact sequence for cohomology of the pair $(X,D)$ with rational coefficients gives: 
  $$ H^{1}(X)\to H^{1}(D)\to H^{2}(X,D)\to H^{2}(X)\to H^{2}(D)\to H^{3}(X,D)\to 0$$
  
  which gives by duality (using $H_1(S)=0$)
  $$0\to H^{1}(D)\to H_{2}(S)\to H^{2}(X)\to \oplus_{i=1}^{r} H^{2}(D_{i})\to 0$$
  
  This shows that $H^{2}(X)$ surjects onto $H^{2}(D)$ and $$\text{rank\ }H_{2}(S) = \text{rank\ }H^{2}(X)-r + \text{rank\ }H^{1}(D)$$ i.e. $b_{2}(S)\leq b_{2}(X)+b_{1}(D)-1$.  But $b_{1}(D)\leq (d-1)(d-2)$ (Lemma 2.4) which gives finally $b_{2}(S)\leq (d-1)^{3}$. 
  
  \end{proof} 
    
    This finishes the argument for smooth $X$. 
    
 \begin{remark}
 More generally, it is known that if $S$ is normal and $X$ is smooth outside $S$ then $\pi_1(S)=(1)$. \cite{Dimca}
 \end{remark}

 \section{Normal completion $X\subset\IP^3$}
 
 Now assume that $S$ is a normal affine surface in $\IA^3$ and its closure $X$ in $\IP^3$ is normal. In this case $X$ has only finitely many singularities whose number we prove is bounded by a function of $d$ and the various characteristics of any singularity like multiplicity, Milnor number, Tjurina number, etc. are also bounded by functions of the degree $d$. 
 
 Let $P$ be the set of singular points of $X$. We show first that cardinality of $P$ is bounded by a function of $d$. Let $e(\cdot)$ denote the topological Euler characteristic. 

 We continue with the assumption that $b_1(S')=0$ for a resolution of singularities $S'$ of $S$. 

\begin{lemma}\label{b3is0} $b_{3}(X)=0$ and $H^{3}(X;\IZ)\cong \textrm{Torsion\ } H_{2}(X;\IZ)$. 
 \end{lemma}
 
 \begin{proof} By our assumption $b_1(S')=0$, where $S'$ is a resolution of singularities of $S$. It follows that for a resolution of singularities $X'$ of $X$ we have $b_1(X')=0$. By duality, $b_3(X')=0$. Since $b_3$ is a birational invariant for normal projective surfaces by Lemma~\ref{b3birinv}, $b_3(X)=0$. 
 
 By the Universal Coefficient Theorem $H^3(X;\IZ)\cong Hom(H_3(X;\IZ),\IZ)\oplus \text{Torsion\ } H_2(X;\IZ)$. Since $b_3(X)=0$ we get the second part of the assertion. 
 
 \end{proof}
 
 \begin{lemma} The number of singular points of $X$ is bounded by $3(d-1)^3$ and the second Betti number $b_{2}(X)$ is  bounded by $d^3-4d^{2}+6d-2$. 
 \end{lemma} 
 
 \begin{proof} Let $X_{0}$ be a smooth hypersurface of degree $d$. Then $e(X_{0})$ is completely determined by $d$, namely $e(X_0)=d^3-4d^{2}+6d$. Now $e(X)=e(X_{0}) - \sum_{p\in P}{\mu_{p}}$ where $p\in{P}$ runs through all singularities and $\mu_{p}$ is the Milnor number of the singular point $p$, \cite{Dimca1986}. It follows that $e(X)\leq d^3-4d^{2}+6d$. 
 
 Since $b_3(X)=0$ by Lemma ~\ref{b3is0}, $e(X)=1+b_{2}(X)+1\leq e(X_{0})$, it also follows that $b_{2}(X)$ is bounded by a function of $d$: $$b_{2}(X)\leq d^3-4d^{2}+6d-2$$
 
 By a theorem of Liu (\cite{Liu}), $\mu_{p}\leq 3\tau_{p}$ where $\tau$ is the Tjurina number. Since $X$ is a surface of degree $d$ its partials w.r.t. local affine coordinates have degrees $\leq d-1$. So we get by Bezout's theorem that $\sum_{p\in{P}} \tau_{p}\leq (d-1)^{3}$. This shows that the number of singular points is bounded by a cubic polynomial in $d$, namely $3(d-1)^3$.
 
 \end{proof}

 \begin{lemma} $H^{3}(X;\IZ)$ is finite. 
 \end{lemma} 
 
 \begin{proof} This follows from the above lemma since the integral homology of $X$ is finitely generated since $X$ is complete.
 
 \end{proof} 
 
 \begin{remark}
 It can be shown that the integral homology of any algebraic variety is finitely generated.
 \end{remark}
 
 We show next that irreducible components of $D'$ are independent in homology of $X'$ where $X'$ is a resolution of singularities of $X$ and $D'$ is the total inverse image of $D=X-S$. We first prove:

  \begin{lemma}\label{bounded determinant} The intersection matrix $(D_{i}'\cdot D_{j}')$ is non-singular.
    \end{lemma}
  
  \begin{proof} If this is not true then there is an integral divisor $D_0'=\Sigma a_iD_i'$ such that $D_0'\cdot D_i'=0$ for all $i$. Since $S$ is affine $D'$ supports a nef and big divisor
  for $X'$ by Goodman's criteria \cite{Goodman}. Hence by Hodge Index Theorem either $D_0'^2<0$, or $D_0'$ is rationally equivalent to $0$. Since $D_0'^2=0$ it follows that $D_0'$ is rationally equivalent to $0$. This gives a non-constant regular invertible function on $S'$. But $b_1(S')=0$ so such a function cannot exist since adjoining $n^{th}$- root of the non-constant unit gives abelian unramified covers of arbitrarily large degree. 

  This contradiction shows that the determinant of $(D_{i}'\cdot D_{j}')$ is non-zero.   
  \end{proof}
  
   \newpage 
   \nind
   The following lemma is slightly more general than the previous results.
   
 \begin{lemma} The intersection matrices of exceptional curves of a resolution of singularities of $X$ are bounded by $d$ for a normal projective surface $X$ of degree d in $\IP^n$. 
 
 \end{lemma}
 
 \begin{proof} We first analyse the situation for an affine normal surface $V$ of degree $d$ in $\IA^n$ and use it for $X$. Let $x,y,z,...$ be affine coordinates in $\IA^n$. We can assume that the morphism given by $(x,y),(x,y,z)$ from $V$ to $\IA^2,\IA^3$ resp. are finite and the second morphism is birational to its image in $\IA^3$. This uses just Noether normalization and primitive element theorem. The image of $V$ in $\IA^3$, say $V_0$, can be assumed to be defined by a degree $d$ polynomial $f=z^d+a_1(x,y)z^{d-1}+...+a_d(x,y)=0$ (since a generic linear projection does not change degrees). Clearly, the branch locus $B\subset\IA^2$ of $V \to \IA^2$ is contained in the branch locus of $V_0\to \IA^2$. The latter is given by considering the resultant $R$ of $f$ and its partial w.r.t. $z$. Since degree $a_i$ is at most $d$ for all $i$ the degree of $R$ is bounded by d. It follows that all the irreducible curves in $B$ have degrees, intersection multiplicities, types of singularities, genera, etc. bounded by functions of $d$.\\
 
 \nind
 By abuse of notation we will now denote the branch locus for the Noether normalization of degree $d$, $X\to\IP^2$, by $B$. Hence we can find $m$ blow-ups of $\IC^{2}$ to make the total transform $\tilde B$ of $B$ snc, where $m$ is bounded by a function of $d$. Let $Y$ be the blow-up of $\IC^{2}$ obtained. We take the normalization of the fibre product $V'=\overline{V\times_{\IC^{2}} Y}$. This is a $d$-fold cover of $Y$ with ramification divisor $R'$ in $V'$ (note that not every irreducible curve in $\tilde B$ may be in the branch locus). 
    
Let $B_1,B_2$ be two irreducible curves in $\tilde{B}$ which meet at a point $p$. In a small neighbourhood $U$ about $p$ in $Y$ with coordinates $z_1,z_2$, on $U_0=U-B_1\cup{B_2}$ we get an at most $d$-fold \'etale cover. Since the fundamental group of $U_0$ is $\IZ\oplus\IZ$, the inverse image of $U_0$ in $V'$ has fundamental group an index at most $d$ subgroup $K$ of $\IZ\oplus\IZ$, say of index $t\leq{d}$. Then $K$ has a subgroup $K'$ of the form $t\IZ\oplus{t\IZ}$. But $K'$ corresponds to the cover $(u,v)=(\sqrt[t]{z_1},\sqrt[t]{z_2})$ of $Y$, which is smooth and of degree $t^2\leq{d^2}$. This implies that the covering $V'$ of $Y$ corresponding to $K$ is an abelian quotient singular point, say a quotient of $\IC^2$ by an abelian group $G$. The inverse image of $(B_1\cup B_2)\cap{U}$ in this smooth cover is the union of axes $uv=0$.
    
    Since the map from this cover to $V'$ is Galois, locally analytically, the inverse image of an irreducible branch curve in $V'$ is a quotient of the inverse image of the branch curve in the above $t^2$-fold cover. Hence every irreducible curve in $R'$ is smooth. In other words, for any irreducible component $E_i$ of $\tilde B$ its inverse image in $V'$ is a disjoint union of smooth irreducible curves. $V'$ is a partial resolution of $V$. We observe that the intersection matrix of the weighted dual graph of the exceptional divisor in $Y$ is bounded by $d$.

    \textit{Claim}: $V'$ has at most cyclic quotient singularities with local $\pi_{1}$ of order at most $d^{2}$. 
    
    \textit{Proof of Claim}: This is well-known. Singular points of $V'$ lie over singular points of $\tilde B$ (analytically near a smooth point of $\tilde{B}$, it is a coordinate $z_1=0$ in a polydisc $\Delta$ in $\IC^2$ and since $\pi_1(\Delta-\tilde{B})\cong\IZ$, locally the ramified cover is cyclic, hence the inverse image of $\tilde{B}\cap\Delta$ is smooth; so singularities of $V'$ can only be over singularities of $\tilde{B}$). We saw above that the singularities of $V'$ are abelian quotient singular points, say a quotient of $\IC^2$ by a group $G$ which can be taken to be a "small" subgroup of $GL_2(\IC)$ since for the normal subgroup $H$ of $G$, generated by pseudo-reflections, $\IC^2/H$ is smooth by the Chevalley-Shepherd-Todd theorem ($G$ is "small" if it contains no non-trivial pseudo-reflections). This group $G$, being abelian, can be assumed to be a diagonal subgroup of $GL_2(\IC)$. 
    
    In fact $G$ must be cyclic. If not, then by the structure theorem, $G$ is a direct sum of cyclics $G_1, G_2$ generated by $g_1, g_2$ respectively (only two cyclics suffice since a quotient of $\IZ\oplus\IZ$ is generated by at most two elements). If the orders of $G_i =m_i$ are coprime then $G$ is cyclic, so the orders have a prime $p$ as a common factor. 
    Let $g_i$ be the diagonal matrix with diagonal entries $a_i, b_i$. Consider $g_i^{m_i/p}$. This has order $p$. So both $g_i^{m_i/p}$ have first diagonal entry a primitive $p^{th}$ root of unity. By adjusting the power of $g_2$ we can assume that for suitable powers, $g_1^l$ and $g_2^r$ have first entries which are inverses of each other, so that $g_1^l.g_2^r$ has first entry $1$. The second entry of the product cannot be $1$, since $l$ and $r$ are strictly less than $m_1, m_2$ respectively and $G$ is a direct sum of $G_i$. Hence $g_1^l.g_2^r$ is a pseudo-reflection in $G$. So $G$ is not small.
 
 We deduce that the singular point $p'$ of  $V'$ is cyclic of order at most $d^2$.

\textit{Claim}: Let $E_i$ be an exceptional irreducible curve for $Y\to\IP^2$ and $E_i'$ an irreducible curve in $V'$ lying over $E_i$. Then $E_i'^2\leq d\cdot E_i^2$.

\textit{Proof of Claim}: Note that since $V'$ is singular some of the $E_i'^2$ can be rational. The inverse image of $E_i$ is a disjoint union of smooth irreducible curves. If $h: V'\to Y$ is the morphism then $(h^{*}E_i)^2=d\cdot E_i^2$. We can write $h^*E_i=\Sigma e_{ij}E_{ij}$ for some integers $e_{ij}$. The inverse image of a small open neighborhood of $E_i$ is
a disjoint union of open neighborhoods of $E_{ij}$. Since $E_i^2<0$ it follows that $E_{ij}^2<0$. Since $dE_i^2=\Sigma e_{ij}^2E_{ij}^2$ the result follows.

If we consider the minimal resolution $V''$ of $p'$ then the exceptional divisor is a linear chain of smooth rational curves which is snc and each irreducible component has self-intersection at most $d^2$, and there are at most $d^2$ irreducible exceptional curves. Considering the proper transforms of the exceptional curves for $V'$ to $V$ in $V''$, we deduce using the standard intersection theory that the total exceptional divisor for $V''$ to $V$ has a bounded intersection matrix.\\
The details are as follows. Let $A_1$ be an irreducible curve in $V'$ passing through a singular point $p'$ of $V'$. In $V''$ the proper transform $A_1'$ of $A_1$ meets an end irreducible component of the exceptional divisor over $p'$ transversally in one point, say $E_1,E_2,..,E_s$, and $A_1'$ meets $E_1$. Let $A_1'+m_1E_1+...+m_sE_s$ be the total transform of $A_1$ in $V''$. The total transform has $0$ intersection with each $E_i$. We know that $s$ and $E_i\cdot E_j$ are bounded by $d$. Solving for $m_i$ by Cramer's rule we see that $m_i$ are bounded by $d$. From these observations we deduce that $A_1'^2$ is also bounded by $d$.\\ 

 \end{proof}

 \begin{lemma} Let $X$ be a normal projective surface of degree $d$ in $\IP^n$. Let $D$ be the curve $X\cap \{h_{0}=0\}$ where $h_{0}=0$ is the hyperplane at infinity. Then there is a suitable resolution of singularities $X'$ of $X$ such that the union of proper transform of $D$ and the exceptional divisor is snc and the matrix of their intersections is bounded by $d$. 
 \end{lemma}
 
 \begin{proof} The proof is similar to the proof of last lemma so we give only a sketch. 
 
 Choose general hyperplanes $\{h_{1},h_{2}\}$ in $\IP^{n}$ such that the morphism given by $p\to [h_{0}(p):h_{1}(p):h_{2}(p)]$ from $X\rightarrow \IP^{2}$ is finite of degree $d$. This is just Noether normalization. Then $h_{0}=0$ on X is the pull-back of a line $L_0$ in $\IP^2$.  Let $B$ be the branch curve in $\IP^2$. Then degree $B$ is bounded by $d$ since $B$ is defined by the discriminant of the defining equation of a general projection of $X$ in $\IP^n$. We can find a suitable sequence of blow ups of $\IP^2$, whose number is bounded by a function of $d$, such that the total transform of $B\cup L_0$, is snc.
Let $Y$ to $\IP^2$ be this final surface. Then as seen in the last lemma, the normalized fiber product $Z:= X\times_{\IP^2}Y$ has at worst cyclic quotient singularities which are bounded by $d$, the finitely many curves of interest in $Z$ have snc. Now $Z$ is a partial resolution of $X$.
Further snc resolution of $Z$, say $X'$, will yield exceptional curves whose intersection matrix is bounded by $d$. We finally have that the matrix of intersections of curves in $X'-S$ is bounded by $d$. 

\end{proof}
 
 \begin{corollary} The determinant of the intersection matrix $(D_{i}',D_{j}')$ is bounded by a function of $d$ where $D_{i}'$ are all the curves in $X'-S'$ with $X'$ a resolution of singularity of $X$.\\ 
  \end{corollary}

 \begin{lemma} Irreducible components of $D'$ are independent in homology. 
 \end{lemma}
 
 \begin{proof} The proof is exactly similar to the proof of Lemma 4.5.\\
 \end{proof}

Let $S,S',X,X'$ be as above, where we only assume that $H_1(S';\IZ)$ is finite and $S$ is embedded as a closed subvariety in $\IA^n$. Then $H_1(X';\IZ)$ is also finite.\\

\nind
We will first prove the following.\\

\begin{theorem} The second Betti number of $X'$ is bounded by $d$ for a suitable resolution of singularities $X'$ of $X$. \\
Similarly, $b_2(X)$ is bounded by $d$.
\end{theorem}
  
 \begin{proof} Note that $b_1(X)=0$. So it suffices to show that $e(X')$ is bounded by $d$.
 
As in an  earlier proof let $X\to\IP^2$ be a Noether normalization of degree $d$. Then there is a smooth surfaces $Y$ obtained by blowing up $\IP^2$ at points such that the number of blowups is bounded by $d$ and the total inverse image of the branch locus $B$ for $X\to\IP^2$ in $Y$ is an snc curve $B'$. Let $Z$ be the normalized fiber product $Y\times_{\IP^2}X$. Then $Z$ has at most cyclic quotient singularities whose resolution dual graphs are bounded by $d$. Let $R$ be the inverse image of $B'$ in $Z$. Now $Z-R\to Y-B'$ is finite unramified of degree $d$. Hence $e(Z-R)=d\cdot e(Y-B')$. Both $e(Y),e(B')$ are bounded by $d$. Hence $e(Y-B')=e(Y)-e(B')$ is also bounded by $d$. Similarly, $e(R)$ is bounded by $d$. Thus, $e(Z)=d\cdot e(Y-B')+e(R)$ is bounded by $d$. Since the weighted dual graphs of minimal resolution of singularities $Z'$ of $Z$ are bounded by $d$ we deduce that $e(Z')$ is bounded by $d$. Since $Z'$ is a resolution of singularities of $X$ the result is proved for $X'$.\\

The proof for $X$ is similar. As seen earlier, $e(B)$ is bounded by $d$. Now $X-\tilde B\to \IP^2-B$ is finite unramified of degree $d$, where $\tilde B$ is the inverse image of $B$ in $X$. Then as in the above argument we deduce that $e(X)$ is bounded by $d$. 

\end{proof}

Now let $S,X,X',D'$ be above. Consider the relative cohomology sequence for the pair $(X',D')$.
It has the portion $H^2(X';\IZ)\to H^2(D';\IZ) \to H^3(X',D';\IZ)\to H^3(X';\IZ)\to (0)$ since $D'$ has real dimension $2$. By duality, $H^3(X';\IZ)\cong H_1(X';\IZ)$. Every irreducible component $D_i'$ of $D'$ gives an element of $H^2(X';\IZ)$. Considering the intersection of this cohomology class with all $D_j'$ we see that the cokernel of $H^2(X';\IZ)\to H^2(D';\IZ)$ has order at most the absolute value of the determinant of the matrix $(D_i'\cdot D_j')$. This latter is bounded by $d$. This implies that the order of $H_1(S';\IZ)$ is bounded by $d$ if and only if the order of $H_1(X';\IZ)$ is bounded by $d$.\\

\nind
We now come to an important result for the proof of the Main Theorem.\\

Consider the class $\mathcal F$ of smooth irreducible surfaces $Y$ of degree $d$ in $\IP^n$ such that $b_1(Y)=0$. Then $H_1(Y;\IZ)$ is finite for all such $Y$.\\

\begin{theorem} There are only finitely many orders of $H_1(Y;\IZ)$ for $Y\in{\mathcal F}$.

\end{theorem}

\begin{proof} Consider the Chow variety $C_{n,2,d}$ of surfaces of degree $d$ in $\IP^n$. There is a closed subvariety $\Gamma\subset C_{n,2,d}\times\IP^n$ which determines the algebraic family of all closed degree $d$ surfaces in $\IP^n$. This gives a morphism $\psi: \Gamma\to C_{n,2,d}$ such that the fibers of $\psi$ are surfaces of degree $d$ in $\IP^n$. By Lemma \ref{Varchenko} there is a stratification of $C_{n,2,d}$ say $C_1\supset C_2\dots\supset C_s$ where $C_1$ is Zariski open in $C_{n,2,d}$ and $C_{i+1}$ is Zariski open in $C_{i}$ such that any two surfaces lying over the points in $C_i - C_{i+1}$ are homeomorphic. A surface $Y$ of degree $d$ with $b_1(Y)=0$ may not lie in the stratum $C_1 - C_2$ necessarily but will lie in some stratum $C_i - C_{i+1}$. The connected components $\{U_{i,i+1}^{j}\}_{j=1,\dots,s_i}$ of this stratum are finitely many and are such that surfaces lying over points in a fixed component $U_{i,i+1}^j$ are mutually homeomorphic. Hence all the surfaces lying over points in $U_i$ have isomorphic $H_1$. This proves the result

\end{proof}

\section{Completion of the proof of Main Theorem}

As above we consider the Chow variety $C_{n,2,d}$ of surfaces of degree $d$ in $\IP^n$ and the algebraic family $\Gamma\subset C_{n,2,d}\times\IP^n$ with the morphism $\psi: \Gamma\to C_{n,2,d}$. Let $\widetilde{\Gamma}\to {\Gamma}$ be a resolution of singularities and $\tilde\psi: \widetilde{\Gamma}\to C_{n,2,d}$ the induced morphism. Now the fibers of $\tilde\psi$ are smooth for points in a non-empty Zariski-open subset $U\subset C_{n,2,d}$. By the earlier argument the fibers lying over a connected component of $U$ are mutually diffeomorphic, and so have isomorphic $H_1$. Next we consider the morphism $\psi: \psi^{-1}(C_{n,2,d}-U)\to C_{n,2,d}-U$. By considering various irreducible components of $C_{n,2,d}-U$ and resolutions of singularities of the inverse images of these irreducible components we deduce that there are only finitely many orders of $H_1(X';\IZ)$ where $X'\to X$ is a resolution of singularities of the completion $X$ of $S$ as in the statement of Main Theorem.\\

Finally, the map $H_1(S';\IZ)\to H_1(S;\IZ)$ is a surjection by Lemma~$2.1$. Hence the order of $H_1(S;\IZ)$ is also bounded by $d$.\\

\nind
Therefore the proof of the main result stated in the introduction is complete.\\

\begin{remark} Using Mumford's presentation for the local fundamental group \cite{Mumford} the same proof shows as above that if a singular point $p$ of $X$ is a quotient singular point then the order of the local fundamental group at $p$ is bounded by a function of $d$.

\end{remark}

\section{Other bounds}

Let $X$ be an irreducible normal projective surface of degree $d$ in $\IP^n$. 

\begin{lemma} The embedding dimension of $X$ at any point is at most $d+1$.
\end{lemma}

\begin{proof} Let $p\in X$ be a closed point and $(R,M)$ the local ring of $X$ at $p$. For two general hyperplane sections $H_1,H_2$ in $\IP^n$ passing through $p$ let their equations at $p$ be $h_1,h_2$ respectively. Since $X$ is a normal surface, it is Cohen-Macaulay and the length of $R/(h_1,h_2)$ is the multiplicity of the point $p$ in $X$. The multiplicity at $p$ is $\leq d$ since $X$ is of degree $d$, therefore length of $R/(h_1,h_2) \leq d$. This implies that $M$ is generated by $d+1$ elements (not necessarily minimally).

\end{proof}

\nind
Now assume that $X$ is a normal surface of degree $d$ in $\IP^3$. We first prove the following for the germ of a singularity on $X$.\\


\begin{lemma} Let $X'$ be a minimal resolution of singularity of the germ $(X,p)$. Then $\dim H^1(X',{\mathcal O_{X'}})$ is bounded by $d$.
\end{lemma}

\begin{proof}
Laufer has proved the following nice formula for the Milnor number $\mu$ of $X$ at a singular point $p$ \cite{Laufer}.

Let $X'$ be a minimal resolution of singularity of the germ $(X,p)$. Then
$$1+\mu=e(E)+K_{X'}\cdot K_{X'}+12\cdot \dim H^1(X',{\mathcal O_{X'}}).$$
Here $E$ is the exceptional divisor of the resolution, $K_{X'}$ the canonical divisor of $X'$, and ${\mathcal O_{X'}}$ the structure sheaf of $X'$. We have seen earlier that $e(E)$ and the determinant $(E_i\cdot E_j)$ are bounded by $d$. Since $X'$ is a minimal resolution and $X$ is Gorenstein
there is a canonical divisor $K_{X'}=\Sigma a_iE_i$ supported on $E$. The arithmetic genera  $p_i$ of $E_i$ are bounded by $d$. By adjunction $E_i^2+K_{X'}\cdot E_i=2p_i-2$. Solving for $a_i$ using Cramer's rule we deduce that $a_i$ are bounded by $d$. This implies that $K_{X'}^2$ is bounded by $d$.
\end{proof}

\begin{lemma} Let $X$ be a normal projective surface of degree $d$ in $\IP^n$ such that $b_1(X')=0$ for a resolution of singularities $X'\to X$. Assume that $X'$ is a minimal resolution of singularities of $X$. Then $p_g(X'),K_{X'}^2$ are bounded by $d$.

\end{lemma}

\begin{proof} We will first prove the result when $X$ is smooth and indicate the modification necessary when $X$ is normal.\\

We have already shown that $b_2(X)$ is bounded by $d$. By Hodge decomposition it follows that
$2\cdot p_g(X)< b_2(X)$.\\
By Noether' formula, $\chi(X,{\mathcal O}_X)=(K^2+e(X))/12$. By assumption, $b_1(X)=0$. So $K^2$ is bounded by $d$.\\

Now assume that $X$ is normal and $f: X'\to X$ a minimal resolution of singularities of $X$.
It is well-known \cite{Artin} that $e(X)-e(X')=\dim R^1f_*{\mathcal O}_{X'}$. We have already proved that the number of singularities of $X$ is bounded by $d$. Also, for each singular point $p$ of $X$ the dimension of the stalk of $R^1f_*{\mathcal O}_{X'}$ at $p$ is bounded by $d$.
It is already proved that $b_2(X')$ is bounded by $d$. Now we can show as in the previous part that $p_g(X'),K_{X'}^2$ are bounded by $d$. 
\end{proof}

\nind
We need the definition of fundamental group at infinity and first homology at infinity for stating the following lemma:

\begin{definition}
    As in our notations, let $S$ be a normal affine surface with completion $X$ and boundary divisor $D$ and let $X'$ be an snc resolution of singularities of $X$. Let $D'$ be the connected inverse image of $D$ in $X'$. Let $N$ be a sufficiently small compact tubular neighbourhood of $D'$ with boundary $\partial{N}$. The fundamental group at infinity $\pi_1^\infty(S)$ is defined to be $\pi_1(\partial{N})$ and $H_1^\infty(S)$ is defined to be $H_1(\partial{N})$. These two groups are independent of the choice of snc resolution of singularities $X'$ hence are well defined.  
\end{definition}

\begin{lemma}
Let $S$ be a normal affine surface of degree $d$ in $\IA^n$ with completion $X$ in $\IP^n$. Suppose that for a resolution of singularities $S'$ of $S$ we have $b_1(S')=0$. Then $H_1^{\infty}(S)$ has rank and order of torsion subgroup bounded by $d$.
\end{lemma}

\begin{proof}
As before, let $X'$ be the minimal snc resolution of singularities of $X$ and let $D'$ be the connected inverse image of $D$ in $X'$. Let there be $p$ cycles in the dual graph of $D'$ and let $\{g_i\}_i$ be the various genera of the irreducible curves in $D'$. Let $N$ be a sufficiently small compact tubular neighbourhood of $D'$ with boundary $\partial{N}$. Then the fundamental group at infinity $\pi_1^{\infty}(S)$ is $\pi_1(\partial{N})$ and $H_1^{\infty}(S):=H_1(\partial{N})$. Mumford showed in \cite{Mumford} that $H_1(\partial{N})=K\oplus\IZ^{p+2\sum_i{g_i}}$ where $K$ is a finite abelian group with $|K|=|det(D'_i.D'_j)|$ where $D'_i$ are the irreducible curves in $D'$. Since we know the boundedness of this determinant by Lemma ~\ref{bounded determinant}, the result on torsion order is proved.

Now consider the long exact sequence for rational cohomology of the pair $(N,\partial{N})$: 
$$...H^1(N)\to H^1(\partial{N})\to H^2(N,\partial{N})...$$
which gives the exact sequence (using duality and that $N$ retracts to $D$ and that $N-\partial{N}$ strong deformation retracts to $D$): 
$$...H^1(D)\to H^1(\partial{N})\to H_2(D)...$$

We have already seen that $b_1(D)$ and $b_2(D)$ are bounded by $d$ hence $b_1(\partial{N})$ is bounded by $d$. Therefore the number of cycles $p$ in the dual graph of $D$ is also bounded by $d$. Since the genera $g_i$ are also bounded by $d$, we get finally that even the rank of $H_1^{\infty}(S)$ is bounded by $d$. This finishes the proof.  

\end{proof}

\begin{lemma}
$H_1(S_{reg})$ is bounded by $d$, where $S_{reg}$ is the smooth locus of a normal affine surface $S\subset\IA^n$ of degree $d$ satisfying $b_1(S')=0$. 
\end{lemma}
 
\begin{proof}
Let $E$ be the total exceptional locus for the map $S'\to{S}$. Then as we have seen $H^2(E)$ is bounded by $d$. Consider the long exact cohomology sequence (with integer coefficients) with compact support for the pair $(S',E)$.
$$H^2_c(S')\to H^2(E)\to H^3_c(S',E)\to H^3_c(S')\to (0).$$
By duality, $H^3_c(S')\cong H_1(S')$. We have already proved that $H^2(E)$ and $H_1(S')$ are bounded by $d$. It follows that $H^3_c(S',E)\cong H_1(S'-E)$ is bounded by $d$.
\end{proof}

\section{Intrinsic Degree}

\begin{definition} For a smooth projective surface $X$ let $d_{intr}$ be the smallest positive integer such that $X$ has a smooth irreducible ample divisor $D$ with base point free linear system $|D|$ and such that $D^2=d_{intr}$. Then three general sections of $|D|$ give a N\"{o}ether normalization $X\rightarrow\IP^2$ of degree $d_{intr}$. We call $d_{intr}$ the intrinsic degree of $X$. 
\end{definition}

We will show that if $f: X'\rightarrow{X}$ is a blow-up at a general point $p\in{X}$ with exceptional curve $E$, the intrinsic degree of $X'$ is bounded by $d$. We will use the following two results by K\"{u}chle and Coppens respectively. 

\begin{lemma}[K\"{u}chle \cite{Kuchle}]
Let $H'=f^*H$ be the pull back of an ample divisor $H$ on $X$ in the above situation. Then the divisor $L=nH'-E$ is ample if and only if $L^2>0$ provided $n\geq{3}$ which can be relaxed to $n\geq{2}$ if $H^2\geq{2}$. 
\end{lemma}

\begin{proposition}
Intrinsic degree of $X'\leq{4d-1}$ where $d\geq{2}$ is the degree of a non-degenerate smooth projective surface $X$ in $\IP^N$. 
\end{proposition}

\begin{proof} We apply K\"{u}chle's result to our situation where $H$ is a hyperplane section of $X$ not containing the point $p$ where we blow-up. First of all $H^2=d>1$ hence we get that $L=2H'-E$ is an ample divisor. Since $L^2=4H'^2-1=4d-1$ we get by definition that $d_{intr}\leq{4d-1}$, i.e. the intrinsic degree of $X'$ is bounded above by a function of the degree of $X$.
\end{proof}

\begin{lemma}[Coppens \cite{Coppens}]
Let $f:X'\rightarrow{X}$ be the blow-up at a general point $p\in{X}$ and suppose $t\geq{7}$ is such that $\dim{\Gamma(X,tH)}\geq{8}$. Then $tH'-E$ is very ample on $X'$ where $H'=f^{*}H$. 
\end{lemma}

\begin{remark}We can apply Coppens' result directly to get a bound on $d_{intr}$ similar to the above Proposition. Since $tH'-E$ is very ample we get $d_{intr}(X')\leq{t^2{d}-1}$ which gives the same result of boundedness of intrinsic degree of $X'$, but with larger bound since $t\geq{7}$.  
\end{remark}

\section{Examples}

\begin{enumerate}

\item Let $S:=\{xy=1\}\subset \IA^3$. Then $S$ is a smooth affine surface of degree $2$, but $b_1(S)=1$.\\

\item Let $S:=\{x^3+y^3+z^3=0\}\subset\IA^3$. Then $S$ is contractible of degree $3$, but for a resolution of singularities $S'\to S$ we have $b_1(S')=2$.\\

\item This example is from \cite{Maharana}. Let $S:=\{z^d-h(x)(h(x)y+1)=0\}$ where $h(x)=x(x-1)(x-2)\dots(x-m+1)$ and $m\leq d$. The morphism $S\to\IA^1$ given by $x$ has fiber $\{x=c\} \cong\IA^1$ for $c\neq 0,1,2,\dots,m-1$. The fibers over $0,1,2,\dots,m-1$ are isomorphic to $\IA^1$ with multiplicity $d$. From this it can be shown that $H_1(S;\IZ)\cong (\IZ/(d))^m$. 

\end{enumerate}

\nind
The first two examples show that the hypothesis $b_1(S')=0$ in the main theorem cannot be removed. The third example shows that the order of $H_1(S;\IZ)$ for a degree $2m+1$ surface can range all the way from a linear function in $d$ to an exponential in $d$. This suggests that no nice bound can be given for a general such surface of degree $d$. \\ 

 \section*{Acknowledgements}
 The authors would like to thank Sudarshan Gurjar for suggesting the proof of Theorem 4.11 using stratification.
 
 The first author is thankful to the Indian National Science Academy for a Fellowship during this work.

 We thank the referee for reading the manuscript carefully and making suggestions for improving the presentation.

 \end{document}